\documentclass[reqno, 11pt]{amsart}
\usepackage[utf8]{inputenc}
\usepackage{color}
\usepackage{calc,amsfonts,amsthm,amscd,epsfig,psfrag,amsmath,amssymb,enumerate,graphicx}
\usepackage[showlabels,sections,floats,textmath,displaymath]{}

\reversemarginpar
\newlength\fullwidth
\numberwithin{equation}{section}

\DeclareMathSymbol{\leqslant}{\mathalpha}{AMSa}{"36} 
\DeclareMathSymbol{\geqslant}{\mathalpha}{AMSa}{"3E} 
\DeclareMathSymbol{\eset}{\mathalpha}{AMSb}{"3F}     
 \def\1{\ifmmode {1\hskip -3pt
    \rm{I}} \else {\hbox {$1\hskip -3pt \rm{I}$}}\fi}

\newtheorem{theorem}{Theorem}[section] 
\newtheorem{lemma}[theorem]{Lemma} 
\newtheorem{proposition}[theorem]{Proposition} 
\newtheorem{rem}[theorem]{Remark} 
\newtheorem{corollary}[theorem]{Corollary} 
\newtheorem{remark}[theorem]{Remark}

\newcommand{\bP}{{\bf P}} 
\newcommand{\bE}{{\bf E}}

\newcommand{\bbE}{{\ensuremath{\mathbb E}} }

\newcommand{\bbN}{{\ensuremath{\mathbb N}} } 
 
\newcommand{\bbP}{{\ensuremath{\mathbb P}} } 
 
\newcommand{\bbR}{{\ensuremath{\mathbb R}} }

\newcommand{\bbZ}{{\ensuremath{\mathbb Z}} }

\newcommand{\gep}{\varepsilon}

\newcommand{\tf}{\textsc{f}}

\newcommand{\ind}{{\bf 1}}

\newcommand{\dd}{\mathrm{d}}

\title[Intermediate phase for oriented percolation]
{Existence of an intermediate phase for oriented percolation}

\author{Hubert Lacoin}
\address{CEREMADE, Place du Mar\'echal De Lattre De Tassigny
75775 PARIS CEDEX 16 - FRANCE}
\email{lacoin@ceremade.dauphine.fr}

\begin{document}

\begin{abstract}
We consider the following oriented percolation model of $\bbN\times \bbZ^d$:
we equip $\bbN\times \bbZ^d$ with the edge set $\{[(n,x),(n+1,y)] | n\in \bbN, x,y\in \bbZ^d\}$, and we say that 
each edge is open
with probability $p f(y-x)$ where $f(y-x)$ is a fixed non-negative compactly supported 
function on $\bbZ^d$ with $\sum_{z\in \bbZ^d} f(z)=1$ and $p\in [0,\inf (1/f)]$ is the percolation 
parameter. 
Let $p_c$ denote the percolation threshold and $Z_N$
the number of open oriented-paths of length $N$ starting from the origin, 
and investigate the growth-rate of $Z_N$ 
when percolation occurs.  
We prove that for if $d\ge 5$ and the function $f$ is sufficiently spread-out, then there exists a second threshold 
$p_c^{(2)}>p_c$ such that $Z_N/p^N$
decays exponentially fast for $p\in(p_c,p_c^{(2)})$ and does not so when $p> p_c^{(2)}$.
The result should extend to the nearest neighbor-model for high-dimension, 
and for the spread-out model when $d=3,4$. It is known (see \cite{cf:Bertin, cf:Lac}) that this phenomenon 
does not occur in dimension $1$ and $2$.
\end{abstract}

\maketitle

\section{Introduction}
Oriented percolation was introduced by Hammersley \cite{cf:H} as a model for porous-media,
the oriented-character of the model can be seen as an attempt to take into account gravity when considering diffusion 
of liquid the medium (as opposed to ordinary percolation which was introduce at the same period).
Its main object is to study connectivity property of an inhomogeneous lattice in a given orientation.
In this paper we do not focus on the most common question for percolation, which is existence of infinite open path, but 
rather on their abundance. 
The main result we have is that if we consider a graph with a density of open edges barely sufficient 
to create infinite open-path, then the number of open path of a given length starting from the origin is much lower than its expected value. 
This implies (see \cite{cf:Y})
that diffusion does not occur in the medium. This contrasts with what happens when
there is an high density of edge and the transversal dimension is larger than $3$, in which case, the number of open path 
is roughly equal to its expected value when percolation occurs. 
Although we believe that this phenomenon holds with great generality,
with the method developed here, we can prove it only for a spread-out version of the model and possibly for 
very high-dimensional nearest neighbor model. The proof is based on size biasing argument and path counting.

\section{Model and result}

\subsection{The model}
We consider the following oriented independent edge-percolation model on $\bbN\times \bbZ^d$:
\begin{itemize}
\item Consider $f: \bbZ^d \to \bbR^+$ and with finite support and  $p\ge 0$ that satisfies 
\begin{equation}
 \sum_{z\in \bbZ^d} f(z)=1 \text{ and } p\le \inf_{z\in \bbZ^d} (1/f(z)):=p_{\max}. 
\end{equation}

\item We define
$\left(X_{(n,x),(n+1,y)}\right)_{n\ge 0, x,y\in \bbZ^d} $, our percolation environment, to be a field of independent 
Bernouilli random variable with parameter 
\begin{equation}
 \bar p(x-y):=p f(x-y),
\end{equation}
(let $\bbP_p$ and $\bbE_p$ denote the probability distribution and expectation).
 \item We say that an edge $\left((n,x),(n+1,y)\right)$ is \textit{open} if $X_{(n,x),(n+1,y)}=1$, and that an oriented path
$$(S_n)_{N_1\le n\le N_2},\quad  N_1\ge 0,\ N_2\in \bbN\cup\{\infty\}$$ is \textit{open} if all the edges 
$$\left((n,S_n),(n+1,S_{n+1})\right)_{N_1\le n\le N_2-1}$$ 
are open.
 \end{itemize}

The usual aim of oriented-percolation is to investigate the existence of infinite open paths and their properties, and more generally,
the connectivity properties of the oriented-network formed by the open-edges. And also the evolution of these property for fixed 
$f$ when $p$ varies.
In this paper we focus more specifically on two cases for the function $f$:

\begin{equation}\label{example}
\begin{split}
 f(z)&:=\frac{1}{2d}\ind_{|z|_1=1},\\
 f(z)&:=\frac{1}{(2L+1)^d}\ind_{|z|_\infty\le L},
\end{split}
 \end{equation}
(where $|z|_1:=\sum_{i=1}^d |z_i|$, $|z|_\infty:=\max_{ 1\le i \le d} |z_i|$)
The first case is called nearest neighbor oriented-percolation and is the one that has received 
the most interest in the physics literature.
The second case is called spread-out oriented-percolation with range $L$ ($L$ is to be thought as a large integer). 
Spread-out models have been studied by mathematicians for a long time for a technical reason:
whereas considering long-range (but finite) interactions instead of nearest-neighbor ones 
should not change the essential properties of a model, a lot of questions becomes 
easier to solve 
for these models when the range $L$ gets large
(an example of that is the use of long-range model to make the lace expansion work
 for all dimensions above the critical one, see 
\cite{cf:Slade}).
We will study also a generalized version of the spread-out oriented-percolation.

\medskip

One defines $\mathcal P$ to be the event of percolation from the origin

\begin{equation}
\mathcal P:= \{\exists (S_n)_{n\ge 0},\ S_0=0,\ \forall n\ge 0, X_{(n,S_n),(n+1,S_{n+1})}=1\}. 
\end{equation}

One defines the percolation threshold by
\begin{equation}\begin{split}
p_c&:=\inf\{p\ge 0\ |\ \bbP_p(\mathcal P)>0\}\\
&=\sup\{p\ge 0\ |\ \bbP_p(\mathcal P)=0\}.
\end{split}\end{equation}

\medskip

The aim of this paper is to discuss the asymptotics of the number of open path of length $N$ starting from the origin 
when percolation occurs.

Define
\begin{multline}\label{Zn}
Z_N(X)=Z_N:=\#\{\text{open oriented-paths of length } N \text{ starting from the origin}\}
\\:= \#\{S: \{0,\dots, N\}\ | \ S_0=0, S \text{ is open } \}.
\end{multline}
We want to compare the asymptotic behavior of $Z_N$ with the one of its expected value: When percolation does not occur, 
$Z_N=0$ for $N$ large enough when percolation occurs from the origin the questions we want to answer are:

\begin{itemize}
\item [(i)] is $Z_N$  asymptotically equivalent (up to a random positive constant)
to $ \bbE_p[Z_N]=p^N$.
\item [(ii)] has $Z_N$ the same exponential growth-rate that $\bbE_p [Z_N]$, i.e.\
is $\log Z_N$ equal to $N(1+o(1))\log p$. 
\end{itemize}

In order to better formulate these questions  we need to introduce some notation and technical results (that we prove in the next Section):

\subsection{Upper-growth rate and renormalized partition function}
Define 
\begin{equation}\begin{split}
W_N(X)&=W_N:=Z_N/\bbE_p[Z_N]=p^{-N}Z_N,\\
\mathcal X&:=\limsup_{N\to \infty} \frac 1 N \log Z_N \le \log p,
\end{split}\end{equation}
(take the convention that the $\limsup$ is equal to $-\infty$ when percolation does not occur) 
that we call respectively the renormalized partition function, and the upper-growth rate for the number of path.

\begin{proposition}[Properties of the upper-growth rate of $Z_N$]
\label{grow}
\begin{itemize}
 \item [(i)] On the event $\mathcal P$, $\mathcal X$ is a.s.\ constant and thus we can defined $\tf(p)$ by the relation
\begin{equation}
\begin{split}
 \mathcal X&=:\tf(p), \quad \bbP_p(\cdot \ | \ \mathcal P) \text{ a.s.} \quad \text{ if } \quad \bbP_p\left(\mathcal P\right)>0,\\
 \tf(p)&:=-\infty, \quad \text{ if } \quad \bbP_p\left(\mathcal P\right)=0.
\end{split}
\end{equation}
 \item [(ii)] The function 
\begin{equation} \label{decres}\begin{split}
[0,p_{\max}]&\to \{-\infty\} \cup \bbR\\
p&\mapsto \tf(p)-\log p.
\end{split}\end{equation}
is non-decreasing so that the threshold  

\begin{equation}\label{pc2}
 p_c^{(2)}:=\inf\{p\in[0,p_{\max}]\ | \ \tf(p)=\log p\}=\sup\{p\in[0,p_{\max}]\ | \ \tf(p)<\log p\}.
\end{equation}
is well defined whenever $\{p\in [0,p_{\max}]\ | \ \tf(p)=\log p\}$ is non-empty.
\end{itemize}
\end{proposition}

\begin{remark}\rm \label{laremkitu}
 It is rather intuitive that some kind of self-averaging should occur and that whenever percolation occurs, one should have
\begin{equation}\label{yte}
 \tf(p)= \lim_{N\to \infty} \frac 1 N \log Z_N. 
\end{equation}
Thus $e^{\tf(p)}$ should be understood as some sort of quenched connective constant for the oriented percolation network.
However we could not prove this statement, nor find a proof of it in the literature. This is the reason why 
the less natural $\limsup$ appears in the definition. 
\end{remark}

\begin{proposition}\label{martin}
The random sequence $(W_N)_{N\ge 0}$  is a positive martingale with respect to the natural filtration 
$(\mathcal F_N)_{N\ge 0}$ defined by
\begin{equation}\label{defFN}
\mathcal F_N:=\sigma(X_{(n,x),(n+1,y)}, x,y\in \bbZ^d, 0\le n\le N-1 ). 
\end{equation}
Thus the limit 
\begin{equation}\label{winf}
 W_\infty:=\lim_{N\to\infty} W_N.
\end{equation}
exists almost surely.
 Condition on $\mathcal P$, $W_\infty$ satisfies the following zero-one law.
 \begin{equation}
 \bbP_p\left( W_\infty>0 \ | \ \mathcal P\right)\in \{0,1\}.
\end{equation}
Moreover the function 
\begin{equation} \label{decres2}
  p\mapsto \bbP_p\left( W_\infty>0\right)
\end{equation}
is non-decreasing so that
\begin{equation}
\label{pc3}
 p_c^{(3)}:=\inf\{p\in[0,p_{\max}] | \ \bbP_p(W_{\infty}>0)>0\}=\sup\{p\in[0,p_{\max}] \ | \ \bbP_p(W_{\infty}>0)=0\},
\end{equation}
is well defined  whenever $\{p\ | \ \bbP_p(W_{\infty}>0)\}$ is non-empty.
\end{proposition}

\begin{remark}\rm
For notational convenience what can set  $p_c^{(2)}=p_{\max}$ (resp. $ p_c^{(3)}=p_{\max}$) when \eqref{pc2} (resp.
\eqref{pc3}) does not give a definition. This is not of crucial importance, since $\{p\ | \ \bbP_p(W_{\infty}>0)\}$ is non-empty 
in all the cases for which we present results.
\end{remark}

With these definition, we trivially have
\begin{equation}
 p_c\le p_c^{(2)}\le p^{(3)}_c\le p_{\max},
\end{equation}
and one wants to investigate if some of these inequality are sharps.

\begin{remark}\rm \label{loc}
Whether $W_N$ tends to a positive limit or tends to zero also has an interpretation in terms of paths localization:
when $p> p_c^{(3)}$ an open path of length $N$ should essentially look like a typical simple random walk path, and, once rescaled,
converge in law to Brownian Motion (to make this a rigorous statement, one needs to adapt the proof of Comets and Yoshida 
\cite{cf:CY}).
On the contrary, when $p<p_c^{(2)}$, trajectories should behave in a totally different way, having different a scaling 
exponent and exhibiting 
localization (in a sense that two independent trajectories chosen at random should intersect many times), 
see \cite{cf:Y2} for a rigorous 
statement about localization for oriented percolation.
\end{remark}

\subsection{Known results from the literature and formulation of the problem}

The idea of investigating $\tf(p)$ appears in the work of Darling \cite{cf:Darling} where a list of open-questions 
is present.

\medskip

A first question is to know what is the condition for having a phase in which $p=\tf(p)$ (to have $p_c^{(2)}<p_{\max}$).
This question has been-fully answered for the nearest-neighbor model and it turns out that this occurs only in dimension $3$ and higher 
(on which we focus our attention)

\begin{equation}
 p_c<p_c^{(2)}=p_c^{(3)}=p_{\max} \quad \Leftrightarrow d=1 \text{ or } 2.
\end{equation}

Yoshida \cite{cf:Y} gave the answer for the case $d=1$ (for a more general model, called Linear Stochastic Evolution) 
adapting a result of Comets and Vargas for directed polymers 
\cite{cf:CV} and also proved $p_c^{(3)}=p_{\max}$ for $d=2$, also using methods related to directed polymers.
Bertin \cite{cf:Bertin} completed the picture by proving $p_c^{(2)}=p_{\max}$ for $d=2$ 
(adapting a result proved for directed polymers in \cite{cf:Lac}).
Note also that these result may be extended into more quantitative one by using the same methods as in \cite{cf:Lac}.
We do not give the proof of these statement as the proof is not short and exactly similar to what is done in \cite{cf:Lac,cf:Bertin2,cf:Lac2}
for directed polymers in discrete a setup, Poissonian environment, or Brownian environment, 
but we think that they are worth being mentioned:

\begin{proposition}\label{fromage}
We have the following quantitative estimate on $\tf(p)$ for $d=1$ and $d=2$
\begin{itemize}
 \item [(i)]
For the nearest neighbor-model with $d=1$, there exists a constant $c$ such that for all $p\in [1,2]$
(recall that $p_{\max}=2$ in that case)
\begin{equation}\label{d1}
\frac{1}{c}(p-p_{\max})^2  \le  \log p -\tf(p)\le  c(p-p_{\max})^2 |\log (p_{\max}-p)|. 
\end{equation}
\item[(ii)]For the nearest neighbor-model with $d=2$ there constants $c$ and $\gep$ such that for all  
 $p\in [4-\gep,4]$ (recall that $p_{\max}=4$ in that case)
\begin{equation}\label{d2}
\exp\left( -\frac c{(p_{\max}-p)^2} \right)  \le  \log p -\tf(p)\le  \exp\left(-\frac{1}{c(p_{\max}-p}\right). 
\end{equation}
\end{itemize}
\end{proposition}

\begin{rem} \rm
 In analogy with what one has for directed polymer the lower bound in \eqref{d1} should give the right order for $\log p -\tf(p)$
whereas the upper-bound is supposed to be the true asymptotic in \eqref{d2}. More justification on 
these conjecture is available in \cite{cf:Lac}.
\end{rem}

In dimension $d\ge 3$, one can to get an upper-bound on $p_c^{(3)}$
by using second-moment method and the fact that $d$-dimensional random walks are transient. 
This was first noticed by Kesten, who was only interested in that to get an upper-bound on the percolation threshold $p_c$,
his proof appears in \cite{cf:CD} for a different oriented percolation model.

\medskip

We present briefly his argument below adapted to our setup:
If  $\sup_{N\ge 0} \bbE_p\left[W_N^2\right]<\infty$, then $W_N$ is uniformly integrable so that
$\bbE_p\left[W_{\infty}\right]=1$ and thus $\bbP_p(W_{\infty}>0)>0$.

\medskip

To compute the second moment one considers $(T_n)_{n\ge 0}$ the random walk on $\bbZ$ starting from zero, 
with IID increments whose law $\bP$ is given by
\begin{equation}\label{ouflemec}
 \bP(T_1=x)=f(x).
\end{equation}
Then 
\begin{equation}\label{secmom}
\bbE_p[W_N^2]=\bE^{\otimes 2}\left[\prod_{\{ n \in[0,N-1] \ | \ T_n^{(1)}=T_n^{(2)} \text{ and } T_{n+1}^{(1)}=T_{n+1}^{(2)}\}}(f(T^{(1)}_{n+1}-T_n^{(1)})p)^{-1} \right].
\end{equation}
where $\bE^{\otimes 2}$ denote expectation with respect to two independent copies of $T$, so that
\begin{equation}\label{condl2}
\lim_{N\to \infty}\bbE_p[W_N^2]<\infty   
\Leftrightarrow
 p\ge \frac{1}{1-\bP^{\otimes 2}\left[S_1^{(1)}\ne S_1^{(2)}\ ; \ \exists n\ge 2,\  T_n^{(1)}=T_n^{(2)} \right]}.
\end{equation}
It implies in particular that $p_c^{(3)}<p_c$ when $f$ is uniformly distributed over 
a set (e.g.\ for both cases given in \eqref{example}).

\medskip
 
For percolation on the $d+1$-regular tree, it is known that $p_c=p_c^{(2)}=p_c^{(3)}$ and this can easily be achieved by performing the two first moment 
of the number of paths.
(see \cite{cf:KS} for the best known general result with that flavor on Galton-Watson trees).

\medskip

We can then precise our question and ask ourselves: when $d\ge 3$ do we have in general $p_c^{(2)}>p_c$, 
which means, do we have an intermediate phase where percolation occurs, but with much less paths than expected, 
and this also on the exponential scale.
We give a (positive) answer to this question in the case of spread-out percolation, in the limit where the range $L$ 
is sufficiently large. The reason why we cannot give a full answer for other cases is our lack of knowledge on the value 
of the percolation threshold
$p_c$, but the method we develop here gives also some results for large $d$ in the 
nearest-neighbor model and for the spread-out model 
with $d=3$ and $4$.
We discuss the case of nearest-neighbor case with $d\ge 3$ small  later in this introduction.

\medskip

The question of of studying the growth rate $\tf(p)$ appears in a work of Darling \cite{cf:Darling}, but it seems that 
it has then been left aside for many years.
In \cite{cf:CPV} a similar problem is raised but concerning the number of path with a density $\rho$ of open edges.
More recently \cite{cf:FukY}, Fukushima and Yoshida proved that $\tf(p)>1$ when $\mathcal P$ occurs with positive probability in the generalized setup of 
Linear Stochastic Evolution.

\subsection{Results}

Our main result in this paper is an asymptotic lower-bound for $p_c^{(2)}$ for both:
\begin{itemize}
 \item [(i)]  The high-dimensional nearest neighbor model (first example in \eqref{example}), for large $d$.
 \item [(ii)] A spread-out model that generalizes the second example in \eqref{example} and that we describe below.
\end{itemize}

Consider $F$ being a continuous function $\bbR^d\to \bbR^+$ 
with compact support which is invariant under the reflections 
$(x_1,\cdots, x_i, \cdots, x_d)\mapsto (x_1,\cdots, -x_i, \cdots, x_d)$ and such that 
$\int F(x)\dd x=1$, and (large) $L\in \bbN$ and set 

\begin{equation}\label{mainmodel}
 f_L(x):=\frac{F(x/L)}{\sum_{y\in \bbZ^d} F(y/L)}.
\end{equation}

\begin{theorem}\label{mainres}
For  the nearest neighbor-model one has, asymptotically when $d\to \infty$ 

\begin{equation}
p_c^{(2)}(d)\ge 1+\frac{\log 2}{2d^2}+o(d^{-2}).
\end{equation}

For the generalized spread-out model one has, for every $d\ge 3$, when $L \to \infty$

\begin{equation}
 p_c^{(2)}(d,L)\ge  1+\log 2\sum_{k=2}^\infty f_L^{\ast 2k}(0)+ O(L^{-3d/2}).  
\end{equation}
where $\ast$ stands for discrete  convolution.
\end{theorem}

For the spread-out model in dimension $d> 4$ the above results gives a positive answer to
the question raised earlier concerning the existence of an intermediate phase : 
\begin{itemize}
\item The bound we have on $p_c^{(2)}$ can also read
\begin{equation}
 p_c^{(2)}\ge  1+\frac{\log 2}{L^d}\sum_{k=2}^\infty F^{\ast 2k}(0)+ o(L^{-d}).
\end{equation}
\item In \cite{cf:VdHS} lace expansion has been used to prove asymptotic in $L$ of $p_c$ 
for the spread-out model for $d>4$.
Their result (Theorem 1.1) implies
\begin{equation}\label{weshwesh}
  p_c^{(2)}=  1+\frac{1}{2L^d}\sum_{k=2}^\infty F^{\ast 2k}(0)+ o(L^{-d}). 
\end{equation}
\end{itemize}
Thus we have
\begin{corollary}
For any $d>4$, for all $L$ large enough one has
\begin{equation}
 p_c^{(2)}(d,L)>p_c(d,L).
\end{equation}
\end{corollary}

In \cite{cf:VdHS}, it is also mentioned that \eqref{weshwesh} should still hold for $d=3,4$ (and thus also the above Corollary). 
Indeed, the analogous of \eqref{weshwesh}
has been proved to hold for the 
contact process in \cite{cf:DP}, and to many respect this model is very similar to oriented percolation.

\begin{remark}\rm
In a work of 
Blease \cite{cf:Bl1}, a heuristic power-expansion of $p_c$ as a function of $1/d$
 is given for a different oriented percolation model closely related to this one. Adapted to our setup, it tells us that
the conjectured asymptotic for $p_c$ is
is 
\begin{equation}\label{coonj}
 p_c= 1+\frac{1}{4d^2}+O(\frac{1}{d^3}).
\end{equation}
Asymptotic for $p_c$ to even higher order have been obtained rigorously for (unoriented) nearest neighbor percolation 
(see \cite{cf:SVdH}), using 
lace-expansion techniques, and it is quite reasonable to think that a similar work for this model 
(which to many respect is simpler to handle than usual percolation) would turn \eqref{coonj} into a rigorous statement, 
and yield
$p_c^{(2)}> p_c$ in high-enough dimension.
\end{remark}
\medskip

\begin{remark}\rm
We believe that $p_c^{(2)}>p_c$ in every lattice model of oriented percolation, when the transversal dimension $d$ is larger 
than $2$.
This conjecture is supported by the fact that we are able to prove the result 
for the spread out model for any profile function $F$.
However there are several obstacles to prove this for the nearest-neighbor model when $d$ is not really large (e.g.\ $d=3$ or $4$). 
We develop this point in the open-question section.
\end{remark}

\subsection{Open questions}

Finally we present open questions or possible direction for research:

\subsubsection{Concerning oriented percolation}:
\begin{itemize}
\item Equation \eqref{condl2} gives an upper-bound for $p_c^{(3)}$,
which for large $L$ gives, (in the spread-out case)
\begin{equation}
 p_c^{(3)}\le 1+\sum_{k=2}^\infty f_L^{\ast 2k}(0)+ O(L^{-2d}).
\end{equation}
This makes us wonder what is the correct asymptotic behavior of $p_c^{(3)}$ at the second order when $L$ goes to infinity.

\item In analogy with a conjecture for directed polymer, it is natural that one should have 
$p_c^{(2)}=p_c^{(3)}$ in general, which means that when $W_N$ decays to zero, it does so exponentially fast, except
maybe at the critical point.
Whether $W_{\infty}=0$ at $p_c^{(3)}$ is more difficult to conjecture and this may depend on the dimension.

\item  Proposition \ref{salutlesamis}, gives a way to get a lower bound on $p_c^{(2)}$ that is quite general
(if $\bbE_p \log \bar Z_{N}> N \log p$, then $p\ge p_c^{(2)}$). This bound should get quite acute when $N$ takes large value.
However $\bbE_p \log \bar Z_{N}$ is quite heavy to compute by brute force as it involves $O(N^d)$ Bernoulli variables.
An interesting perspective would be get
method to compute the above 
expectation in an effective way (e.g. combining computer calculus and concentration-like theoretical results) 
and and see how the obtained bound compares with conjectured values for $p_c$ (\cite[Table |B.1]{cf:SD} ).
To get an answer to the question ``$p_c^{(2)}>p_c^{(1)}$ ?'' one also have to find an good upper-bound on $p_c^{(1)}$.
Indeed the only rigorous upper-bound for $p_c^{(1)}$ we have seen in the literature is the one from \cite{cf:CD} and it is also 
an upper-bound for $p_c^{(3)}$ (and thus can be of no use for our purpose).

\end{itemize}

\subsubsection{Concerning percolation}:
In analogy with what we do here for oriented-percolation, it is  a quite natural question to study the
{\sl quenched} connective constant of percolated lattices, that is to say: the asymptotic growth-rate (with $N$) 
of the number of edge 
(or site) self-avoiding paths of length $N$
starting from a given point of the graph. In particular one would like to understand how its compares with
the growths rate of its expected value
(the {\sl annealed} connective constant which is trivially equal to $p\nu$ where $\nu$ is the connective constant 
for the original lattice). 
The problem seems considerably more difficult that in the directed case, as 
self-avoiding walks are involved instead of directed walks 
(see the monograph of Madras and Slade for a rather complete account on Self Avoiding Walk.
\cite{cf:MSlade}). This issue has been studied in the physics literature 
(see e.g.\ \cite{cf:BKC} and reference therein although we would not quite agree with the conjecture that 
are present there, see below), but to our knowledge, 
no rigorous result has been established so far on the mathematical side:

\begin{itemize}
\item To begin with it would be nice to prove the existence of the quenched connectivity constant 
(something we were not able to 
do in the directed case see Remark \ref{laremkitu}). 
\item Another question would be the existence of a phase where the 
number of open self-avoiding path starting from the origin behave likes its expected value 
(like here when $d\ge 3$ and $p$ large).
This could happens on $\bbZ^d$ if the dimension is large enough,
(say e.g.\ $d>4$ as $4$ is the critical dimension for self-avoiding walk, 
but the author has far from enough evidence to state this as 
a conjecture), 
and one could try to prove that using second moment method similar to \eqref{secmom}:
The question reduces then to know whether or not the Laplace transform of the overlap of 
two infinite self avoiding walk of is bounded in a neighborhood
of zero. This seems quite a difficult question to tackle but maybe not hopeless as there have been quite a lot of tools 
developed to
study and understand the self-avoiding walk in high-dimension 
(see \cite{cf:Slade}).
\item A challenging issue is 
to prove that this never happens in low dimension (say $d=2$ and $d=3$), and that for whatever small edge-dilution,
the quenched connectivity constant is strictly smaller than the annealed one.
As in the directed case, there are heuristics evidence that this happens in dimension $2$ and $3$,
but there is a need of a better (i.e.\ rigorous) understanding the behavior of the self-avoiding walk 
to transform that into a proof. The recent important result 
obtained on the hexagonal lattice \cite{cf:DS} gives some hopes that more about 
that will be known in the future at least in some special two dimensional case.
\item An easier one is to settle whether 
if, in high dimension, just above the critical point, 
the number of path is exponentially smaller than its expected value.
For this, the techniques used in this paper might adapt and  some precise asymptotics in $d$ 
are available for the value of $p_c$ is available in the literature (see \cite{cf:SVdH}).
\end{itemize} 

\section{Technical preliminaries} \label{trivio}

This Section is devoted to the proof of Proposition \ref{grow} and \ref{martin}.
We let the reader check that $W_N$ is a martingale, and prove all the other statements. The Section is divided into two parts,
one for the proof of the $0-1$ law statements, and the other for the proof of monotonicities in $p$ by coupling.


\subsection{Zero-One laws}

In this Section we prove
\begin{itemize}
 \item [(i)] $\bbP_p\left[W_{\infty}>0 \ |  \ \mathcal P\right]\in\{0,1\}$.
 \item [(ii)] $\mathcal X$ is as constant on the event $\mathcal P$.
\end{itemize}

The proof of the two statement use exactly the same ideas thus we prove $(i)$ in full detail and then explain how to adapt 
the proof to get $(ii)$.

\medskip

 Set $\kappa= \bbP_p\left[W_{\infty}>0\right]$. 
For $x\in \bbZ^d$, $N\in \bbN$ set $Z_{N}(x)$ to be the number of open path from $(0,0)$ to $(N,x)$
\begin{equation}
 Z_{N}(x):=\# \{ (S_n)_{n\in [0,N]} \ | \ S_{0}=0, S_{N}=x, S \text{ is open} \}.
\end{equation}
We define also $\theta_{N,x}$ to be the shift operator on the environment $X$ (which determines the set of open edges) by
\begin{equation}
\theta_{N,x}(X)_{(n,u),(n+1,v)}:=X_{(n+N,u+x),(n+N+1,v+x)}
\end{equation}
and $\theta_{N,x} W_\infty$ to be limit of the renormalized partition 
function constructed from the shifted environment $\theta_{N,x}(X)$ instead of $X$.
One has
\begin{equation}
 W_{\infty}:= \sum_{x\in \bbZ^d}p^{-N}Z_{N}(x) \theta_{N,x} W_\infty.
\end{equation}
Then, as $(\theta_{N,x} W_\infty)_{x\in \bbZ^d}$ is independent of $\mathcal F_N$ (recall \eqref{defFN}),
\begin{equation}
 \bbP_p( W_{\infty}>0 \ | \mathcal F_N)=
\bbP_p\left[\exists x, Z_{N}(x)>0,  \theta_{N,x} W_\infty>0 \ | \ \mathcal F_N \right]\ge \kappa\ind_{Z_N>0}. 
\end{equation}
Thus 
\begin{equation}
  \bbP_p( W_{\infty}>0 \ | \ \mathcal F_N)\ge \kappa\ind_{Z_N>0},
\end{equation}
so that making $N$ tends to infinity, one gets
\begin{equation}  
\ind_{\{W_{\infty}>0\}}\ge \kappa\ind_{\mathcal P},
\end{equation}
which is enough to conclude that either $\kappa=0$ or $W_{\infty}>0$ on the event of percolation.

%
%
%
%
%
%
%
%
%
%

\medskip

We now turn to $(ii)$. 
Given $r$ define $\bbP_p\left[\mathcal X> r\right]:=\kappa'$.
Then by the same reasoning as for the proof of $(i)$

\begin{equation}
 \bbP_p\left[ \mathcal X>   r \ | \mathcal F_N\right]\ge \kappa'\ind_{\{Z_N>0\}}.
\end{equation}

Thus making $N$ tends to infinity one get that either $\kappa'=0$ or that $\mathcal X> r$ a.s.\
when $\mathcal P$ occurs, and this for all $r$.

\qed

\subsection{Existence of the threshold}

We prove now that the function in \eqref{decres} and \eqref{decres2} are non-decreasing (and thus that $p_c^{(2)}$ and $p_c^{(3)}$ are well defined)
by a coupling argument.
The idea of the proof is quite common in percolation, it is to 
couple the realization of the process for different $p$ in a monotone way. 
 We consider a field of random variables $(U_{(n,x),(n+1,y)})_{n\ge 0, x,y \in \bbZ^d}$ uniformly distributed on $[0,1]$ (denote its law by $\bbP$).
 Then one sets 
 \begin{equation}
  X_{(n,x),(n+1,y)}(p):= \ind_{\{U_{(n,x),(n+1,y)}\le f(y-x) p\}}.
 \end{equation}
It is immediate to check that $  X_{(n,x),(n+1,y)}(p)$ has distribution $\bbP_p$.
This construction implies immediately that $\bbP_p(\mathcal P)$ is a non-decreasing function of $p$ and existence of $p_c$ (but this is an already well established fact).
In this section we write $W_N(p)$ for $W_N(X(p))$.
\medskip

Set \begin{equation}
  \mathcal F_p:= \sigma\left( X_{(n,x),(n+1,y)}(p), n\in \bbN, x,y \in \bbZ^d \right)
 \end{equation}

Our key observation is that for $p'\le p$

\begin{equation}\label{voiture}
 \bbE\left[W_N(p') \ | \  \mathcal F_p\right]=W_N(p). 
\end{equation}

Indeed the probability that a path is open for $X(p')$ knowing that it is open for $X(p)$ is equal to $(p'/p)^N$.
This gives immediately  
\begin{equation}
 \bbP\left[W_\infty(p')>0\right]\le \bbP\left[W_\infty(p)>0\right].
\end{equation}

\medskip

To check that 
$\tf(p')-\log p'\le \tf(p)-\log p$, it is sufficient to show that for all $\gep$, for all sufficiently large $N$
\begin{equation}\label{fthew}
 \log W_N(p')\le N(\tf(p)-\log p+\gep).
\end{equation}
For $N$ large enough $W_N(p)\le N(\tf(p)-\log p+\gep/2)$, by the definition of $\tf(p)$ and furthermore by \eqref{voiture} and the Borel-Cantelli Lemma.
\begin{equation}
  \log W_N(p')\le \log W_N(p)+N\gep/2,
\end{equation}
so that \eqref{fthew} holds.
\qed

%
%
%
%
%
%
%
%
%
%
%
%
%
%
%
%
%
%
%
%
%
%
%
%
%
%

\section{Size biasing}

A key element of our proof is to consider the system under a law which has been tilted by  $W_N$: 
the size-biased version of $\bbP_p$.
In this section we present a nice way to describe the law of the environment under the size biased law, encountered in a paper of 
Birkner on directed polymers \cite{cf:Birk}, and much related to the so-called spine techniques used in the study of branching structures 
(see e.g. \ \cite{cf:LPP}).
Then we use this construction to get some operational condition under which the inequality $\tf(p)<\log p$ holds.
The argument developed here are completely general and can be used for any kind of directed percolation model.

\subsection{A description of the distribution of the environment under the size-biazed law}

We define $\tilde \bbP^N_p$ the so-called size-biased measure on the edge-environment $X$, as a measure absolutely continuous w.r.t to $\bbP_p$ 
and whose Radon-Nikodym derivative is given by
\begin{equation}
\frac{\dd \tilde \bbP^N_p}{\dd\bbP_p}(X):=W_N(X),
\end{equation}
and then one studies the behavior of $W_N$ under this new measure.
There are some reason for doing so, e.g.\ the sequence of $W_N$ converges to zero in law  if and only if 
$W_N\to \infty$ in law under the size biased measure $\tilde \bbP^N_p$.

\medskip

We give a nice representation of the size-biased measure, adapted from the work of Birkner \cite{cf:Birk}
on directed polymers.
First we sample a random walk $T$ chosen according to probability measure $\bP$
given by \eqref{ouflemec}.
Then given a realization of $X$ under $\bbP_p$, 
we consider $\tilde X$ an alternative percolation environment whose definition is given by

\begin{equation}
\tilde X_{(n,x),(n+1,y)}:=\begin{cases}
                           1 \quad \text{if } T_n=x,\ T_{n+1}=y,\\
X_{(n,x),(n+1,y)} \quad \text{otherwise}.
                          \end{cases}
\end{equation}

Let $\tilde Z_N$, $\tilde W_N$ be defined analogously to $Z_N$ and $W_N$ but using environment 
$\tilde X$ instead of $X$. Then $\tilde W_N$ under the law $\bbP_p \otimes \bP$ as the same law that 
$W_N$ under the size biased measure $\tilde \bbP_p$. More precisely

\begin{proposition}
For any function $F: \bbR \to \bbR$ one has
\begin{equation}
 \bbE_p\bE\left[F(\tilde Z_N)\right]=\tilde \bbE^N_p\left[F(Z_N)\right]:=\bbE_p\left[W_N F(Z_N)\right].
\end{equation}
(the same result being obviously valid when $Z_N$ is replaced by $W_N$)
\end{proposition}

\begin{proof}
 See \cite{cf:Birk} Lemma 1.
\end{proof}

\subsection{A link between the size-biased law and the original law}

Some important properties of the law $W_N$ under the original measure can be recovered from 
its property under the size biased measure. 
For example if $W_{N}$ tends to $0$ whereas its expectation is equal to 
one for every $N$, it means that $W_N$ must be large on some atypical event that carries most of the expectation.
Under the size-biased measure, this event must become typical.

\begin{proposition}
We have the following properties concerning $\tilde W_N$,

\begin{itemize}
  \item [(i)] If under $\bbP_p\otimes \bP$,
\begin{equation}
 \lim_{N\to \infty} \tilde W_N=\infty, 
\end{equation}
in law, then
\begin{equation}
 W_{\infty}=0, \bbP_p \text{ a.s.}
\end{equation}

 \item [(ii)] If there exists a constant $c>0$ such that for all $N$ large enough
\begin{equation}\label{expotruc}
  \bbE_p\bE\left[\tilde W_N\le N \right]\le e^{-cN},
\end{equation}
then
\begin{equation}\label{expotruc2}
 \tf(p) < \log p-c/2.
\end{equation}

\end{itemize}
\end{proposition}

\begin{proof}
 The first point is classic. If $W_{\infty}$ is non degenerate, there exists constants $a>0$ and $A>a$ such that 
$\bbP_p \left[W_N\in (a,A)\right]>c$ uniformly for all $N$.
Then $\tilde \bbP^N_p\left(W_N \in [a,A]\right)\ge ac$ so that in can't converge to $\infty$ in law.

\medskip

For the second point, if \eqref{expotruc} is valid then $W_{\infty}=0$ from the first point, 
so that once $N$ is large enough $W_N\le 1$.
We just have to check that $W_N\in (e^{-cN/2} ,1) $ does not happen infinitely often. One has
\begin{equation}
\bbP_p\left[W_N\in (e^{-cN/2} ,1)\right]=\tilde\bbE^N_p\left[\frac{1}{W_N}\ind_{W_N\in (e^{-cN/2} ,1)} \right]
\le e^{cN/2}\tilde\bbP_p\left[W_N\in (e^{-cN/2} ,1) \right]\le e^{-cN/2}.
\end{equation}

Thus by Borel-Cantelli Lemma $\bbP_p$-a.s.\ eventually for all $N$ 
\begin{equation}
 W_N \le e^{-cN/2}.
\end{equation}

\end{proof}

\subsection{Reduction to a finite volume criterion}

The criterion given by \eqref{expotruc}, will be satisfied if one can bound 
$\log \tilde W_N$ from below by a sum of independent variable 
that have positive mean. This way of doing gives us a very simple criterion for having $\tf(p)<\log p$.

Set 

\begin{equation}
 \bar Z_N::= \#\{S: \{0,\dots, N\}\ | \ S_0=0, \ S_N=T_N, S \text{ is open for $\tilde X$} \}.
\end{equation}

It is straight-forward to see that 

\begin{equation}
  \bar Z_{N+M}\ge \bar Z_N\times \bar Z_M^{(1)},
\end{equation}
where $ \bar Z_M^{(1)}$ is independent of $\bar Z_{N}$ and has the same law that $\bar Z_M$.
Hence we have

\begin{proposition}\label{salutlesamis}
If for some value of $N_0$ one has 
\begin{equation}\label{missmay}
 \bbE_p\bP \log \bar Z_{N_0}> N_0 \log p,
\end{equation}
 then there exists a constant $c$ such that \eqref{expotruc} holds (and thus do does \eqref{expotruc2}).
\end{proposition}

\begin{proof}

For $N\ge 0$, set $N=:nN_0+r$ be the Euclidean division of $N$ by $N_0$.
One has

\begin{equation} \label{maminova}
 \log \bar Z_N\ge \sum_{i=1}^n \log \bar Z_{N_0}^{(i)},
\end{equation}
where $ \bar Z_{N_0}^{(i)}$ are independent copies of $Z_{N_0}$.
As a sum of i.i.d.\ bounded random variables, the probability of large deviation of the r.h.s.\ of 
\eqref{maminova} below its average is exponentially small in $n$ (thus in $N$) and this ends the proof.
\end{proof}

\section{Proof of Theorem \ref{mainres}}

Now we want to use the criterion provided by Proposition \ref{salutlesamis} to get lower bounds for
$p_c^{(2)}$ for specific models, either high-dimensional nearest-neighbor or spread-out. The particularity of 
these model is that all the $p^{(i)}_c$ are really close to one.

Our strategy is just to estimate
the probability that $\bar Z_N\ge 2$  to get a lower-bound on $ \frac{1}{N}\bE \bbE_p \log \bar Z_N$.

We start with the spread-out model

\begin{lemma}

For the spread-out model with $d\ge 3$ fixed, for all $p\ge 1$, 
\begin{equation}
  \frac{1}{N}\bE \bbE_p \log \bar Z_N \ge \log 2 \sum_{k=2}^N \frac{N-k}{N} f^{\ast 2k}(0)(1+O(N L^{-d})).
\end{equation}
In particular choosing $N_L:=L^{d/2}$ one gets 
\begin{equation}
 \frac{1}{N_L}\bE \bbE_p \log \bar Z_{N_L} \ge \log 2 (1+O(L^{-d/2}))\sum_{k=2}^N  f^{\ast 2k}(0).
\end{equation}

\end{lemma}

As a consequence of the Lemma, one gets that \eqref{missmay} holds as soon as 
\begin{equation}
 p\ge 1+\log 2 \sum_{k=2}^N  f^{\ast 2k}(0)+O(L^{-3d/2}).
\end{equation}
and this gives the wanted lower-bound for $p_c^{(2)}$ of Theorem \ref{mainres}.

\begin{proof}

One can focus on the case $p=1$ with no loss of generality.
Moreover one can bound the expectation of $\log \bar Z_N$ as follows, 
\begin{equation}
 \bE \bbE_1 \log \bar Z_N \ge  \bE\bbP_1 \left[\bar Z_N\ge 2\right]\log 2.
\end{equation}

Then notice that $Z_N\ge 2$ if and only if there exists $a<b$ and an open path linking
$(a,T_a)$ and $(b,T_b)$ that does not use any edge nor meet any sites on the path $T$ (except the starting and ending sites).
Given a fixed $T$ we want to estimate $\bbP_1 \left[\bar Z_N\ge 2\right]$.
We call a path linking $(a,T_a)$ and $(b,T_b)$ a bridge (see Fig. \ref{bridge}), we call $\mathcal B_T=\mathcal B$ the set of ``bridges'' on $(T_n)_{0\le N}$
\begin{multline}
\mathcal B_T:=\{(B_n)_{a\le n \le b}\ | \ \exists 0\le a<b \le n, 
\\ \ B_a=T_a, B_b=T_b, \forall c\in (a,b), B_c\ne T_c, f(B_{c+1}-B_c)>0\} 
\end{multline}

\begin{figure}[hlt]
\begin{center}
\leavevmode 
\epsfxsize =14 cm
\psfragscanon
\psfrag{O}{$O$}
\psfrag{N}{$N$}
\psfrag{a}{$a$}
\psfrag{b}{$b$}
\epsfbox{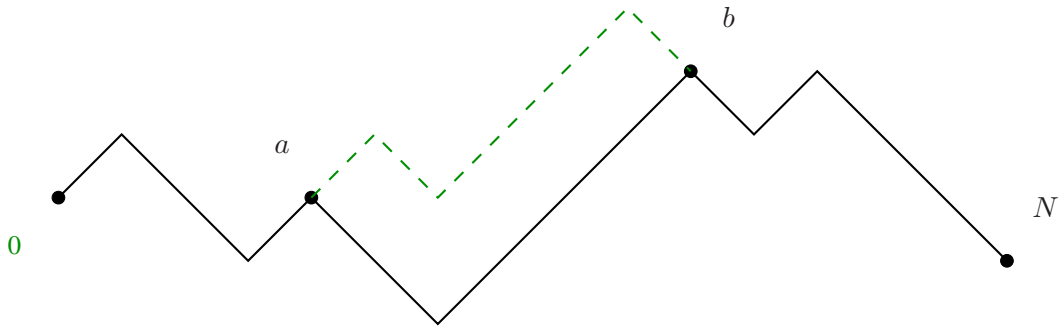}
\end{center}
\caption{\label{bridge} Illustration of a trajectory $T$ of length $N=15$ (full line), together with a bridge between $a=4$ and $b=10$
for the nearest-neighbor model in dimension $1+1$. The bridge does not meet the trajectory $T$ (site-wise) except at the starting and ending points. 
}
\end{figure}
One has

\begin{multline}\label{dirlo}
 \bbP_1 \left[\bar Z_N\ge 2\right]=\bbP_1\left[\exists B\in \mathcal B\ ; \ B \text{ is open }\right]
\\
\ge \sum_{B\in \mathcal B } \bbP_1\left[B \text{ is open }\right]
- \sum_{B \mathcal B }\bbP_1\left[ B \text{ is  open } \ ; \ \exists \tilde B\in \mathcal B\setminus \{ B\}, \tilde B \text{ is open }\right]
\\
\ge 
\left(\sum_{B\in \mathcal B } \bbP_1\left[B \text{ is open }\right]\right)\left(1-\max_{B\in \mathcal B} 
\bbP_1\left[ \exists \tilde B\in \mathcal B\setminus \{ B\}, \tilde B \text{ is open } \ | \  B \text{ is open}\right]\right).
\end{multline}

We first control the term $\sum_{B\in \mathcal B } \bbP_1\left[B \text{ is open }\right]$ and its expected value with respect to $T$.
It is larger than
\begin{equation}\label{dirlo2}
  \sum_{a=0}^{N-2}\sum_{b=a+2}^{N}
\left( f^{\ast (b-a)}(T_b-T_a)-\sum_{c=a+1}^{b-1}f^{\ast (c-a)}(T_c-T_a)f^{\ast (b-c)}(T_c-T_b)\right),
\end{equation}
(in the first term, some paths between $(a,T_a)$ and $(b,T_b)$ intersect $T$ at intermediate point, i.e.\ that are not bridges, have been counted, 
the second term is subtracting the contribution of all those path: all path intersecting $T$ at 
time $c$ with varying $c$.
Some contribution are subtracted more than once and that is the reason why we get an inequality).

Using the fact that $f$ is symmetric we can compute the expected value of the first term in \eqref{dirlo2} (averaging over $T$ that has law $\bP$). 
It is equal to 
\begin{equation}
\sum_{a=0}^{N-2}\sum_{b=a+2}^{N} f^{\ast 2(b-a)}(0)
= \sum_{k=2}^N (N-k) f^{\ast 2k}(0).
\end{equation}
The expected value with respect to $\bP$ of the second term in \eqref{dirlo2} is smaller than

\begin{equation}
 \sum_{a=0}^{N}\left(\sum_{b=0}^{\infty} f^{\ast (2b)}(0)\right)^2=O(NL^{-2d})
\end{equation}

Now it remains to show that for all choices of $B$ and $T$,
\begin{equation}
\bbP_1\left[ \exists \tilde B\in \mathcal B\setminus \{ B\}, \tilde B \text{ is open } \ | \  B \text{ is open}\right] \text{ is small }.
\end{equation}

\medskip

To get an additional bridge on $T$ between $a$ and $b$ knowing than $B$ is open, one can either have an open path not using edges of 
$T$ and $B$ that links $(a,T_a)$ to $(b,T_b)$ or use open edges 
of $B$ to construct a new open bridge say
by having an open path that links say $(a,B_a)$ and $(b,T_b)$ (or $(a,T_a)$ and $(b,B_b)$, or $(a,B_a)$ and $(b,B_b)$).
We can use union bound to get that

\begin{multline}
 \bbP_1\left[ \exists \tilde B \in \mathcal B\setminus\{B\}\  \tilde B \text{ is open } | B \text{ is open } \right]\\
\le \sum_{a=1}^{N-1}\sum_{b=a+1}{N}p^{b-a}
\sum f^{\ast (b-a)}(T_b-T_a)+f^{\ast (b-a)}(B_b-T_a)\\
+f^{\ast (b-a)}(T_b-B_a)+f^{\ast (b-a)}(T_b-B_a)
\\
\le 4 \sum_{a=0}^N \sum_{n=1}^\infty \max_z  f^{\ast n}(z) \le O(N L^{-d}).
\end{multline}

The conclusion of all this is that
\begin{equation}
 \bE \bbP_1 \left[\bar Z_N\ge 2\right]=
 \left(\sum_{k=2}^N (N-k) f^{\ast 2k}(0)+O(NL^{-2d})\right)(1-O(NL^{-d})).
\end{equation}
And thus that
\begin{equation}
  \bE \bbE_1 \left[\log \bar Z_N\right]\ge \log 2 \left( \sum_{k=2}^N (N-k) f^{\ast 2k}(0)+O(N^{2}L^{-2d})\right).
\end{equation}

\end{proof}

\medskip

We can turn to the nearest neighbor case which is a bit simpler

\begin{lemma}

For the nearest neighbor we have the following lower asymptotic in $d$, valid for all $p\ge 1$, and for all $N$ 
\begin{equation}
   \bE\bbE_p [\log \bar Z_N ]\ge \log 2 \left[1-\left[1-\frac{1}{2d^2}+O(d^{-3}) \right]^N\right].
\end{equation}
In particular choosing $N=d$ one gets 
\begin{equation}
 \frac{1}{d}\bE \bbE_p \log \bar Z_{d} \ge \frac{\log 2}{2d^2}+O(d^{-3}).
\end{equation}

\end{lemma}

As a consequence of the Lemma,  by monotonicity in $p$ one gets that \eqref{missmay} holds as soon as 

\begin{equation}
 p< \exp \left( \frac 1 d \bE \bbE_1 \log  \bar Z_{d} \right).
\end{equation}
and this gives the wanted lower-bound for $p_c^{(2)}$ of Theorem \ref{mainres}.

\begin{proof}
We consider only the case $p=1$.

We use the same strategy as for the spread-out model, but here we only need to consider the bridges of length $2$.
Set $(e_1,\dots, e_d)$ the canonical base of $\bbZ^d$ ( $e_i:=(0,\dots,0 ,1,0,\dots,0)$ where $1$ lies in the $i$-th position) and $(e_{d+1},\dots, e_{2d}):=-(e_{d},\dots, e_{d})$

Given a realization of $T$ (which in this case is a nearest neighbor simple random walk on $\bbZ^d$), the number of possibily to have a bridge of length $2$ between $a-1$ and $a+1$ 
depends on the local configuration of $T$ (see Fig. \ref{bridge2}):

\begin{itemize}
 \item $(i)$ If $T_a-T_{a-1}=T_{a+1}-T_a$ then there is no possibility. 
 \item $(ii)$ If $T_a-T_{a-1}=-T_{a+1}+T_a$ (i.e.\ if $T_{a-1}:=T_{a+1}$) 
then there are $2d-1$ possibilities for having a bridge, each has probability $(1/2d)^2$:
 opening the edges $[(a-1,T_{a-1}),(a,T_a)]$ and $[(a,T_{a}+e_i),(a+1,T_{a+1})]$ where $i$ is such that $e_i\ne T_{a}-T_{a-1}$.
 \item $(iii)$ In all other cases, there is only one possibility which is opening the edges $[(a-1,T_{a-1}),(a,T_{a-1}+e_{i})]$ 
and $[(a,T_{a-1}+e_i), (a+1,T_{a+1})]$ where $e_{i}=T_{a+1}-T_a$.
\end{itemize}
Then we note that all the bridges over $T$ length $2$ are pairwise edge-disjoint, so that given $T$ each 
bridges are open independently with probability $(2d)^{-2}$ (recall that $p=1$).

\begin{figure}[hlt]
\begin{center}
\leavevmode 
\epsfxsize =14 cm
\psfragscanon
\epsfbox{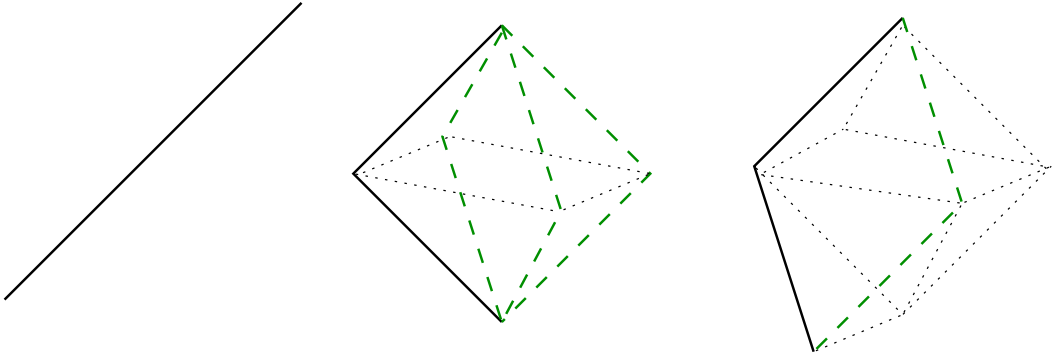}
\end{center}
\caption{\label{bridge2}
Schematic representation of the three different possibilities listed in the $1+2$ dimensional case, trajectories are oriented along the vertical direction.
The full line represent the portion of trajectory and the dashed ones, the potential bridges of length $2$. 
From left to right: $(i)$ If the two increments for the walk are equal, there is no possibility to build a bridge, $(ii)$
if they are opposite, one has $3=2*2-1$ options for bridges, $(iii)$ if the two increments are along different dimension, there is only one way to build a bridge (by inverting 
the order of these increments).}
\end{figure}

Thus

\begin{multline}
 \bbP_1 [\bar Z_N = 1]\le \bbP_1\left[\text{ All bridges of length $2$ are closed }\right]\\
\le \left(1-(2d)^{-2}\right)^{(2d-1)\#\{a\in [1,N-1] \ | \ T_{a-1}=T_{a+1}\}+
\#\{a\in [1,N-1]  \ |  \ T_{a+1}-T_{a}\ne \pm T_{a}-T_{a-1} \}}.
\end{multline}

Averaging with respect to $T$ gives

\begin{equation}
  \bE\bbP_1 [\bar Z_N = 1]\le \left[ \frac 1 {2d} \left(1-(2d)^{-2}\right)^{(2d-1)}+\frac {d-1} d  \left(1-(2d)^{-2}\right)
+\frac 1 {2d}\right]^{N-1} =\left[1-\frac{1}{2d^2}+O(d^{-3}) \right]^N.
\end{equation}
Thus

\begin{equation}
   \bE\bbE_1 [\log \bar Z_N ]\ge \log 2 \left[1-\left[1-\frac{1}{2d^2}+O(d^{-3}) \right]^N\right].
\end{equation}
\end{proof}

\medskip

{\bf Acknowledgement:} The author is very grateful to N.\ Yoshida for suggesting this problem, and enlightening discussions, and to
D.\ Villemonais for valuable comments on the manuscript.

This work was written during the author long-term stay in IMPA, he acknowledges kind hospitality and support.

\end{document}